\documentclass[12pt]{article}
\usepackage{amsfonts}
\newtheorem{lemma}{Lemma}
\newtheorem{theorem}{Theorem}
\newtheorem{proposition}{Proposition}

\newenvironment{proof}{{\bf Proof}.}{\hfill $\Box$}

\begin{document}


\title{\sc Self-adjointness and boundedness in quadratic quantization\footnote{Research partially supported by PBCT-ADI 13 grant ``Laboratorio de An\'alisis y Estoc\'astico'', Chile}}

\author{Ameur Dhahri \\\vspace{-2mm}\scriptsize
Department of Mathematics\\\vspace{-2mm}\scriptsize  Faculty of Sciences of
Tunis\\\vspace{-2mm}\scriptsize
University of Tunis El-Manar, 1060 Tunis, Tunisia\\\vspace{-2mm}\scriptsize
E-Mail: {\tt ameur@volterra.uniroma2.it}}
\date{}
\maketitle
\begin{abstract}
We construct a counter example showing, for the quadratic quantization, the identity $(\Gamma(T))^*=\Gamma(T^*)$ is not necessarily true. We characterize all operators on the one-particle algebra whose quadratic quantization are self-adjoint operators on the quadratic Fock space. Finally, we discuss the boundedness of the quadratic quantization. 
\end{abstract}
\section{Introduction}

The usual ($1-st$ order) quantization consists in the study of unitary representations of the Heisenberg algebra and leads to the study of the Fock functor. Quadratic quantization consists in the study of unitary representations of $sl(2,\mathbb{R})$ algebra and leads to the study of the quadratic Fock functor. The nonlinearity restricts the set of morphisms to which this functor can be applied. For example while the $1-st$ quantization of an operator $T$ is bounded if and only if $T$ is a contraction, necessary condition for the quadratic quantization of $T$ is bounded is that $T$ is a contration. A neccessary and sufficient condition for the boundedness of the quadratic quantization at the momentum is not known.

In this paper, we characterize all operators on the one-particle algebra such that its quadratic quantization are self-adjoint operators on the quadratic Fock space. Moreover, we discuss the boundedness of the quadratic quantization.

This paper is organized as follows. In section 2, we recall the main properties of the quadratic Fock functor. The characterization of all operators on the one-particle algebra whose quadratic quantization are self-adjoint operators is given in section 3. Finally, in section 4, we discuss the boundedness of the quadratic quantization and we give an example of operator $T$ on the one-particle algebra such $\Gamma_2(T^*)\neq(\Gamma_2(T))^*$.

\section{The quadratic Fock functor}
In this section, we recall some basic definitions and properties of the quadratic exponential vectors and the quadratic Fock space (cf \cite {AcDhSk}, \cite{AcDh1}, \cite{AcDh2}, \cite{Dh}).

\subsection{Quadratic Fock space}

The quadratic Fock space $\Gamma_2(L^2(\Bbb R^d)\cap L^\infty(\Bbb R^d))$ is the closed linear span of $\big\{B^{+n}_f\Phi$, $n\in\Bbb N$, $f\in L^2(\Bbb R^d)\cap L^\infty(\Bbb R^d)\}$, where $B^{+0}_f\Phi=\Phi$, for all $f\in L^2(\Bbb R^d)\cap L^\infty(\Bbb R^d)$. In \cite{AcDh1} it is proved that $\Gamma_2(L^2(\Bbb R^d)\cap L^\infty(\Bbb R^d))$ is an interacting Fock space. Moreover, the scalar product between two $n$-particle vectors is given by the following.
\begin{proposition} \label{prop1}For all $f,\,g\in L^2(\Bbb R^d)\cap L^\infty(\Bbb R^d)$, one has
\begin{eqnarray*}
\langle B^{+n}_f\Phi,B^{+n}_g\Phi\rangle&=&c\sum^{n-1}_{k=0}2^{2k+1}{n!(n-1)!\over((n-k-1)!)^2}\,\langle f^{k+1}, g^{k+1}\rangle\\
\;\;\;\;\;&&\langle B^{+(n-k-1)}_f\Phi,B^{+(n-k-1)}_g\Phi\rangle\\
&=&\!\!\!\!\!\!\!\!\!\!\!\!\sum_{i_1+2i_2+\dots+ki_k=n}\frac{2^{2n-1}(n!)^2c^{i_1+\dots+i_k}}{i_1!\dots i_k!2^{i_2}\dots k^{i_k}}\langle f,g\rangle^{i_1}\langle f^2,g^2\rangle^{i_2}\dots\langle f^k,g^k\rangle^{i_k}.
\end{eqnarray*}
\end{proposition}

The quadratic exponential vector of an element $f\in L^2(\mathbb{R}^d)\cap L^\infty(\mathbb{R}^d)$, if it exists, is given by
$$
\Psi(f)=\sum_{n\geq0}\frac{B^{+n}_f\Phi}{n!}
$$ 
where by definition
\begin{eqnarray*}\label{Psi(0)Phi} 
\Psi(0)= B^{+0}_f\Phi = \Phi.
\end{eqnarray*}
In \cite{AcDh1}, it is proved that the quadratic exponential vector $\Psi(f)$ exists if $\|f\|_\infty<\frac{1}{2}$, and does not exists if  $\|f\|_\infty>\frac{1}{2}$. Besides, the scalar product between two exponential vectors $\Psi(f)$ and $\Psi(g)$ is given by
\begin{equation}\label{Form}
\langle \Psi(f),\Psi(g)\rangle=e^{-\frac{c}{2}\int_{\mathbb{R}^d}\ln(1-4\bar{f}(s)g(s))ds}.
\end{equation}

For the proof of the following theorem, we refer to \cite{AcDh1}.
\begin{theorem}
The quadratic exponential vectors are linearly independents. Moreover, the set of quadratic exponential vectors is a total set in $\Gamma_2(L^2(\Bbb R^d)\cap L^\infty(\Bbb R^d))$.
\end{theorem}

Finally, it is proved in \cite{AcDh2} that $\Gamma_2(L^2(\mathbb{R}^d)\cap L^\infty(\mathbb{R}^d))$ is an interacting Fock space.
\begin{theorem} 
There is a natural ismorphism between the quadratic Fock space
$\Gamma_2(L^2(\mathbb{R}^d)\cap L^\infty(\mathbb{R}^d))$ and the interacting Fock space
$\oplus_{n=0}^{\infty}\otimes^{n}_{symm}\{L^2(\mathbb{R}^d) , \langle \cdot,\cdot\rangle_n\}$,
with scalar product:
$$
\langle f^{\otimes n} , g^{\otimes n}\rangle_n
=\sum_{i_1+2i_2+\dots+ki_k=n}\frac{2^{2n-1}(n!)^2c^{i_1+\dots+i_k}}{i_1!\dots
i_k!2^{i_2}\dots k^{i_k}}\langle f,g\rangle^{i_1}\langle f^2,g^2\rangle^{i_2}
\dots\langle f^k,g^k\rangle^{i_k}
$$ 
\end{theorem}

\subsection{Quadratic quantization}

For all linear operator $T$ on $L^2(\Bbb R^d)\cap L^\infty(\Bbb R^d)$, we define its quadratic quantization, if it is well defined, by 
$$\Gamma_2(T)\Psi(f)=\Psi(Tf)$$
for all $f\in L^2(\mathbb{R}^d)\cap L^\infty(\mathbb{R}^d)$ such that $\|f\|_\infty<\frac{1}{2}$. 

Note that in \cite{AcDh2}, the authors have proved that  if $\Gamma_2(T)$ is well defined on the set of the quadratic exponential vectors, then $T$ is a contraction on $L^2(\mathbb{R}^d)\cap L^\infty(\mathbb{R}^d)$ equipped with the norm $\|.\|_\infty$. Conversely, if $T$ is a contraction on $L^2(\mathbb{R}^d)\cap L^\infty(\mathbb{R}^d)$ equipped with the norm $\|.\|_\infty$, then $\Gamma_2(T)$ is well defined on the set of the quadratic exponential vectors $\Psi(f)$ such that $\|f\|_\infty<\frac{1}{2}$. Moreover, they have characterized the operators $T$ on $L^2(\Bbb R^d)\cap L^\infty(\Bbb R^d)$ whose quadratic quantization is isometric (resp. unitary). The boundedness of $\Gamma_2(T)$ was also investigated.
\section{Self-adjointness of the quadratic  quantization}
 In this section our purpose is to give a necessary and sufficient conditions in order to characterize the self-adjointness of the quadratic quantization.

\begin{lemma}\label{lem2}
Let $T$ be a contraction on $L^2(\Bbb R^d)\cap L^\infty(\Bbb R^d)$ with respect to the norm $\|.\|_\infty$. If $\Gamma_2(T)$ is a self-adjoint operator on $\Gamma_2(L^2(\Bbb R^d)\cap L^\infty(\Bbb R^d))$, then 
\begin{eqnarray}\label{ecuacion}
\langle (T(f))^n,g^n\rangle=\langle f^n, (T(g))^n\rangle,
\end{eqnarray}
for all $f,\,g\in L^2(\Bbb R^d)\cap L^\infty(\Bbb R^d)$ and $n\geq1$.
\end{lemma}
\begin{proof}
Let $f,g\in L^2(\mathbb{R}^d)\cap L^\infty(\mathbb{R}^d)$. Then, there exists $\delta>0$ such that for all $0\leq t\leq \delta$
$$2\sqrt{\delta}\|f\|_\infty<1,\;\,2\sqrt{\delta}\|g\|_\infty<1$$
Since $\Gamma_2(T)$ is a self-adjoint operator on $\Gamma_2(L^2(\mathbb{R}^d)\cap L^\infty(\mathbb{R}^d))$, one has
\begin{eqnarray*}
\langle \Psi(\sqrt{t}T(f)),\Psi(\sqrt{t}g)\rangle=\langle \Psi(\sqrt{t}f),\Psi(\sqrt{t}T(g))\rangle
\end{eqnarray*}
This yields
\begin{eqnarray*}
e^{-\frac{c}{2}\int_{\mathbb{R}^d}\ln(1-4t\overline{T(f)}(s)g(s))ds}=e^{-\frac{c}{2}\int_{\mathbb{R}^d}\ln(1-4t\bar{f}(s)T(g)(s))ds}
\end{eqnarray*}
and
\begin{eqnarray}
\int_{\mathbb{R}^d}\ln(1-4t\overline{T(f)}(s)g(s))ds=\int_{\mathbb{R}^d}\ln(1-4t\bar{f}(s)T(g)(s))ds\label{self}
\end{eqnarray}
Put
$$h_s(t)=\ln(1-4t\overline{T(f)}(s)g(s)),\;h(t)=\int_{\mathbb{R}^d}\ln(1-4t\overline{T(f)}(s)g(s))ds.$$
Then, the $n$--th derivative (in $t$) of $h_s(t)$ is given by
$$
h_s^{(n)}(t)=2^{2n}(-1)^n(n-1)!(\overline{T(f)}(s))^n(g(s))^n(1-4t\overline{T(f)}(s)g(s))^{-n}.
$$
Note that for all $0\leq t\leq\delta$
\begin{eqnarray}\label{y}
|h_s^{(n)}(t)|\leq \frac{2^{2n}(n-1)!|\overline{T(f)}(s)|^n|g(s)|^n}{(1-4\delta\|T(f)\|_\infty\|g\|_\infty)^n}
\end{eqnarray}
for $a.e \;\,s\in\mathbb{R}^d$. Since the left hand side of (\ref{y}) is integrable in $s$, one gets
$$h^{(n)}(t)=\int_{\mathbb{R}^d}h_s^{(n)}(t)ds.$$
This yields
\begin{equation}\label{0derivative}
h^{(n)}(0)=2^{2n}(-1)^n(n-1)!\langle (T(f))^n,g^n\rangle.
\end{equation}
Therefore identity (\ref{ecuacion}) is deduced by taking the derivative of the both member sides in (\ref{self}) and by using identity (\ref{0derivative}).
\end{proof}

The following result follows from \cite{Dix} and \cite{San}.
\begin{lemma}\label{lem3}
Let $E$ be a measurable subset of $\mathbb{R}^d$. Let $T:L^\infty(\mathbb{R}^d)\rightarrow L^\infty(E)$ a continuous homomorphism such that $T(\bar{f})=\overline{T(f)}$ for all $f\in L^\infty(\mathbb{R}^d)$. Then, there exists a measurable function $\varphi:E\rightarrow \mathbb{R}^d$ such that $T(f)=f\circ\varphi$, for all $f\in L^\infty(\mathbb{R}^d)$. 
\end{lemma}
\begin{proof}
Note that there exists an $*-$isomorphism $\Phi:L^\infty(\mathbb{R}^d)\rightarrow L^\infty((0,1))$ (i.e $\Phi$ is a continuous bijective homomorphism which satisfies $\Phi(\bar{f})=\overline{\Phi(f)}$). Then, from \cite{Dix} (cf Appendix IV), there exists a measurable function \\
$\varphi_1:(0,1)\rightarrow\mathbb{R}^d$ such that $\Phi(f)=f\circ\varphi_1$ $a.e$, for all $f\in L^\infty(\mathbb{R}^d)$. Define $T_1=T\circ\Phi^{-1}$. It is clear that $T_1:L^\infty((0,1))\rightarrow L^\infty(E)$ is a continuous homomorphism which satisfies $T_1(\bar{f_1})=\overline{T_1(f_1)}$ for all $f_1\in L^\infty((0,1))$. Lemma 2 in \cite{San} implies that $T_1$ is weak* continuous on $L^\infty((0,1))$. Moreover, from Theorem 1 in \cite{San}, there exists a measurable function $\varphi_2:E\rightarrow (0,1)$ such that $T_1(f)=f\circ\varphi_2$ a.e, for all $f\in L^\infty((0,1))$. Recall that $T=T_1\circ \Phi$. This proves that for all $f\in L^\infty(\mathbb{R}^d)$, $T(f)=f\circ\varphi$ $a.e$, where $\varphi=\varphi_1\circ\varphi_2:E\rightarrow \mathbb{R}^d$ is a measurable function. This ends the proof.  
\end{proof}

The following theorem gives a characterization of all operators $T$ on $L^2(\mathbb{R}^d)\cap L^\infty(\mathbb{R}^d)$ such that its quadratic quantization $\Gamma_2(T)$ is a self-adjoint operator on the quadratic Fock space. 
\begin{theorem}
Let $T$ be a contraction for $L^\infty(\mathbb{R}^d)$ which is a bounded operator on $L^2(\mathbb{R}^d)$. Then, $\Gamma_2(T)$ is a self-adjoint operator if and only if there exist a measurable subset $E\subset\mathbb{R}^d$, a function $h:\mathbb{R}^d\rightarrow\mathbb{C}$, with $\|h\|_\infty\leq1$, and an involutive, measurable, Lebesgue measure preserving function $\varphi: E\rightarrow E$ such that $\bar{h}(x)=h\circ\varphi(x)$ for all $a.e$ $x\in E$ and
$$T(f)=\chi_E\,h\,f\circ\varphi$$
for all $f\in L^2(\mathbb{R}^d)\cap L^\infty(\mathbb{R}^d)$, where $\chi_E$ is the characteristic function on $E$.
\end{theorem}

\begin{proof}
Identity (\ref{ecuacion}) implies that $T$ is a self-adjoint operator on $L^2(\mathbb{R}^d)$. Moreover, one has
\begin{equation}\label{dh}
\langle (T(f_1+f_2))^2,(g_1+g_2)^2\rangle=\langle (f_1+f_2)^2, (T(g_1+g_2))^2\rangle,
\end{equation}
for all $f_1,\,f_2,\,g_1,\,g_2\in L^2(\Bbb R^d)\cap L^\infty(\Bbb R^d)$. Now, using (\ref{ecuacion}) and (\ref{dh}), one gets
\begin{eqnarray*}
\langle T(f_1)T(f_2),g_1g_2\rangle=\langle f_1f_2,T(g_1)T(g_2)\rangle,
\end{eqnarray*}
for all $f_1,\,f_2,\,g_1,\,g_2\in L^2(\Bbb R^d)\cap L^\infty(\Bbb R^d)$. Because $T=T^*$, one has
\begin{eqnarray}\label{sam}
T(f_1)T(f_2)\bar{g}_1=T(f_1f_2\overline{T(g_1)})\;\;\;a.e
\end{eqnarray}
for all $f_1,\,f_2,\,g_1\in L^2(\Bbb R^d)\cap L^\infty(\Bbb R^d)$. Let $h$ be an accumulation point of $T(\chi_{[-n,n]^d})$ with respect to the weak topology of $L^\infty(\Bbb R^d)$. Because $T$ is a contraction on $L^\infty(\Bbb R^d)$, it follows that $\|h\|_\infty\leq1$. Besides, from (\ref{sam}), one has
\begin{eqnarray}\label{pai}
T(f_1)T(f_2)=T(f_1f_2\bar{h}) \;\;\;a.e,
\end{eqnarray}
for all $f_1,\,f_2\in L^2(\Bbb R^d)\cap L^\infty(\Bbb R^d)$. In the same way, for all $f_1\in L^2(\Bbb R^d)\cap L^\infty(\Bbb R^d)$
\begin{eqnarray}\label{san}
T(f_1)h=T(f_1\bar{h})\;\;\;a.e.
\end{eqnarray} 
 Hence, identities (\ref{pai}) and (\ref{san}) imply that for all $f_1,\,f_2\in L^2(\Bbb R^d)\cap L^\infty(\Bbb R^d)$
\begin{eqnarray}\label{mat}
T(f_1)T(f_2)=T(f_1f_2)h\;\;\;a.e.
\end{eqnarray}

Let $V:L^2(\Bbb R^d)\cap L^\infty(\Bbb R^d)\rightarrow L^\infty(\Bbb R^d)$ be the linear operator defined by $$T(f)=hV(f),$$
for all $f\in L^2(\Bbb R^d)\cap L^\infty(\Bbb R^d)$. Then, identity (\ref{mat}) gives
$$h^2V(f_1)V(f_2)=h^2V(f_1f_2)\;\;\;a.e$$
for all $f_1,\,f_2\in L^2(\Bbb R^d)\cap L^\infty(\Bbb R^d)$. Let $E=supp(h)$. One has
$$V(f_1)V(f_2)=V(f_1f_2)\;\;\;a.e\;\mbox{ on }\,E.$$
Then, the operator $V$ can be extended as an homomorphism from $L^\infty(\Bbb R^d)$ to $L^\infty(E)$. Now, let us prove that $V:L^\infty(\mathbb{R}^d)\rightarrow L^\infty(E)$ is a contraction. Let $f\in L^2(\Bbb R^d)\cap L^\infty(\Bbb R^d)$. Recall that from (\ref{mat}) one has
$$T(f)^2=T(f^2)h\;\;\;a.e.$$
Then, by induction, we prove that for all $f\in L^2(\Bbb R^d)\cap L^\infty(\Bbb R^d)$ and for all $n\geq1$
$$T(f)^n=T(f^n)h^{n-1}\;\;\;a.e.$$
It follows that for all $n\geq1$
$$|T(f)|^n\leq \|f\|_\infty^n|h|^{n-1}\;\;\;a.e$$
and
$$|T(f)|\leq \|f\|_\infty|h|^{\frac{n-1}{n}}\;\;\;a.e$$
Thus, by taking the limit $n\rightarrow\infty$, on gets
$$|T(f)(x)|\leq ||f||_\infty|h(x)|,$$
for $a.e$ $x\in\mathbb{R}^d$. This shows that 
$$|h(x)||V(f)(x)|\leq ||f||_\infty|h(x)|,$$
for $a.e$ $x\in\mathbb{R}^d$ and $\|V(f)\|_{L^\infty(E)}\leq ||f\|_\infty$, where
$\|f\|_\infty=\|f\|_{L^\infty(\mathbb{R}^d)}$.

Now, let $f\in L^2(\Bbb R^d)\cap L^\infty(\Bbb R^d)$ such that $\bar{f}=f$. It is clear that for all $t\in\mathbb{R}$
$$V(e^{itf})=e^{itV(f)}.$$
Since $V:L^\infty(\mathbb{R}^d)\rightarrow L^\infty(E)$ is a contraction, then $|e^{itV(f)}(x)|\leq1$ for all $a.e$ $x\in E$ and for all $t\in\mathbb{R}$. This proves that $V(f)(x)\in\mathbb{R}$, for all $a.e$ $x\in E$ and $V(f)$ is a real function. Thus $V(\bar{f})=\overline{V(f)}$ for all $f\in L^\infty(\mathbb{R}^d)$. Therefore, Lemma \ref{lem3} implies that there exists a measurable function $\varphi:E\rightarrow \mathbb{R}^d$ such that 
$$V(f)=f\circ\varphi\;\;a.e$$
for all $f\in L^2(\Bbb R^d)\cap L^\infty(\Bbb R^d)$. By using identity (\ref{san}), one has
$$h^2(x)f\circ\varphi(x)=h(x)\bar{h}\circ\varphi(x) f\circ\varphi(x),$$
for $a.e$ $x\in E$ and for all $f\in L^2(\Bbb R^d)\cap L^\infty(\Bbb R^d)$. This gives
\begin{eqnarray}\label{var}
\bar{h}\circ\varphi(x)=h(x),
\end{eqnarray}
for $a.e$ $x\in E$ and $\varphi(E)\subset E$. Now, let us prove that $\varphi$ is an involutive function. Because $T$ is a bounded operator, then for all $f\in L^2(\mathbb{R}^d)\cap L^\infty(\Bbb R^d)$
$$\|hf\circ \varphi\|_2\leq d\|f\|_2.$$ 
Hence, there exists a function $l:\mathbb{R}^d\rightarrow\mathbb{C}$ bounded by $d=\|T\|_2$ such that
\begin{equation}\label{risz}
\int_E|h(x)|^2f\circ\varphi(x)dx=\int_Ef(x)\bar{l}(x)dx,
\end{equation}
for all $f\in L^2(\mathbb{R}^d)\cap L^\infty(\Bbb R^d)$. Note that
$$\langle T(f),T(g)\rangle=\langle f, T^2(g)\rangle,$$
for all $f,\,g\in L^2(\Bbb R^d)\cap L^\infty(\Bbb R^d)$. This yields
\begin{eqnarray}\label{fi}
\int_E|h(x)|^2\bar{f}\circ\varphi(x)g\circ\varphi(x)dx&=&\int_E h(x)h\circ\varphi(x)\bar{f}(x)g\circ\varphi^2(x)dx\nonumber\\
&=&\int_E|h(x)|^2\bar{f}(x)g\circ\varphi^2(x)dx,
\end{eqnarray}
for all $f,\,g\in L^2(\Bbb R^d)\cap L^\infty(\Bbb R^d)$. Thus identities (\ref{risz}) and (\ref{fi}) gives
\begin{eqnarray*}
\int_E\bar{f}(x)g(x)\bar{l}(x)dx=\int_E|h(x)|^2\bar{f}(x)g\circ\varphi^2(x)dx,
\end{eqnarray*}
for all $f,\,g\in L^2(\Bbb R^d)\cap L^\infty(\Bbb R^d)$. This implies that
\begin{eqnarray}\label{igual}
|h(x)|^2g\circ\varphi^2(x)=g(x)\bar{l}(x),
\end{eqnarray}
for $a.e$ $x\in E$ and for all $g\in L^2(\Bbb R^d)\cap L^\infty(\Bbb R^d)$. Let $g_n(x)=e^{-\|x\|^n}$, $n\geq1$. It is clear that $g_n\in L^2(\Bbb R^d)\cap L^\infty(\Bbb R^d)$. Moreover, one has
\begin{eqnarray}\label{lim}
l(x)=e^{\|x\|^n-\|\varphi^2(x)\|^n}|h(x)|^2,
\end{eqnarray}
for $a.e$ $x\in E$ and for all $n\geq1$. Put
\begin{eqnarray*}
J_+&=&\{x\in E;\; \;\;\|x\|>\|\varphi^2(x)\|\},\\
J_-&=&\{x\in E;\; \;\;\|x\|<\|\varphi^2(x)\|\}.
\end{eqnarray*}
Suppose that, the Lebesgue measure of $J_+$, $|J_+|>0$. Then, identity (\ref{lim}) implies that
$$l(x)=\lim_{n\rightarrow+\infty} e^{\|x\|^n-\|\varphi^2(x)\|^n}|h(x)|^2=+\infty,$$
for $a.e$ $x\in J_+$, which is a contradiction with the fact that $l$ is a bounded function on $\mathbb{R}^d$. In the same way if $|J_-|>0$, we prove that $l(x)=0$ for $a.e$ $x\in J_-$. Then, by using (\ref{lim}), one gets $h(x)=0$ for $a.e$ $x\in J_-\subset E$. This is a contradiction with the fact that $J_-\subset E=supp(h)$. Thus, for $a.e$ $x\in E$, one has $\|x\|=\|\varphi^2(x)\|$ and $l(x)=|h(x)|^2$. Furthermore, identity (\ref{igual}) implies that $g(x)=g\circ\varphi^2(x)$ for $a.e$ $x\in E$ and for all $g\in L^2(\mathbb{R}^d)\cap L^\infty(\mathbb{R}^d)$. This shows that $\varphi^2(x)=x$ for $a.e$ $x\in E$.

Now, let $m$ be the density of $\lambda\circ\varphi$ where $\lambda$ is the Lebesgue measure. Because $T=T^*$, one has
\begin{eqnarray}\label{T=T*}
\int_Ef(x)\bar{h}(x)\bar{g}(\varphi(x))dx=\int_Eh(x)f(\varphi(x))\bar{g}(x)dx,
\end{eqnarray}
for all $f,\,g\in L^2(\Bbb R^d)\cap L^\infty(\Bbb R^d)$. Note that
\begin{equation}\label{chgvar}
\int_Ef(x)\bar{h}(x)\bar{g}(\varphi(x))dx=\int_Ef(\varphi(x))\bar{h}(\varphi(x))\bar{g}(x)m(x)dx.
\end{equation}
Then, identities (\ref{T=T*}) and (\ref{chgvar}) imply that $h(x)=\bar{h}(\varphi(x))m(x)$ for $a.e$ $x\in E$. Since $h(x)=\bar{h}(\varphi(x))$ for $a.e \,\,x\in E$, one gets $m(x)=1$ for $a.e\,\,x\in E$. This ends the proof.
\end{proof}

\section{On the boundedness of the quadratic quantization}

Recall that the Fock functor has its origin in Heisenberg commutation relations. So, if $\mathcal{H}$ is a complex Hilbert space, the Heisenberg algebra $Heis(\mathcal{H})$ is generated by
$$\{A_f,\;A^+_f, \;1;\;\;\;f\in \mathcal{H}\}$$
with commutation relations
$$[A_f,A^+_g]=\langle f,g\rangle1,\:\:\:f,g\in\mathcal{H}$$
(the omitted commutation relations are zero) and involution
$$(A_f)=A^+_f,\;\;\;f\in\mathcal{H}$$
The Fock representation of $Heis(\mathcal{H})$ is characterized by a cyclic vector $\Omega$ such that $A_f\Omega=0$, for all $f\in\mathcal{H}$. The bosonic Fock space $\Gamma_s(\mathcal{H})$ is the closed linear span of 
$$\{A_f^{+n}\Omega,\;\;\;f\in \mathcal{H}\}$$
The quantization of an operator $T$ on $\mathcal{H}$ is defined by
$$\Gamma_1(T)\varepsilon(f)=\varepsilon(Tf)$$
where $\varepsilon(f)=\sum_{n\geq0}\frac{A_f^{+n}\Omega}{\sqrt{n!}}$ for all $f\in\mathcal{H}$. It is well known that:
\begin{enumerate}
\item[1)] $\Gamma_1(T)$ is a bounded operator on $\Gamma_s(\mathcal{H})$, iff $\Gamma_1(T)$ is a contraction on $\Gamma_1(T)$, iff $T$ is a contraction on $\mathcal{H}$,
\item[2)] $\Gamma_1(T^*)=\Gamma_1(T)^*$ for all contraction $T$ on $\mathcal{H}$.
\end{enumerate}

In the quadratic case, we prove that $T$ is a contraction on $L^2(\mathbb{R}^d)$ is a necessary condition in order that $\Gamma_2(T)$ is a contraction on $\Gamma_2(L^2(\Bbb R^d)\cap L^\infty(\Bbb R^d))$. Moreover, we give a counter-example of contraction $T$ on $L^2(\Bbb R^d)\cap L^\infty(\Bbb R^d)$ with respect to $\|.\|_\infty$ such that property $2)$ is not satisfied. However, $1)$ is still an open problem.
\begin{lemma}\label{lem4}
For all $f_1,\dots,f_n\in L^2(\Bbb R^d)\cap L^\infty(\Bbb R^d)$ and $\alpha_1,\dots,\alpha_n\in\mathbb{C}$, we have
\begin{equation}\label{necessary}
c\|\alpha_1f_1+\dots+\alpha_n f_n\|^2=\frac{d}{dt}\Big|_{t=0}\|\alpha_1\Psi(\sqrt{t}f_1)+\dots+\alpha_n\Psi(\sqrt{t}f_n)\|^2
\end{equation}
\end{lemma}

\begin{proof}
For all $0\leq t\leq\delta$ such that $\sqrt{\delta}\sup_{1\leq i\leq n}\|f_i\|_\infty<\frac{1}{2}$, we have
\begin{eqnarray}\label{deriv}
\|\alpha_1\Psi(\sqrt{t}f_1)+\dots+\alpha_n\Psi(\sqrt{t}f_n)\|^2\!\!\!&=&\sum_{i,j=1}^n\bar{\alpha_i}\alpha_j\langle\Psi(\sqrt{t}f_i),\Psi(\sqrt{t}f_j)\rangle\nonumber\\
&=&\!\!\!\sum_{i,j=1}^n\bar{\alpha_i}\alpha_je^{-\frac{c}{2}\int_{\mathbb{R}^d}\ln(1-4t\bar{f_i}(s)f_j(s))ds}
\end{eqnarray}
So, in the same way as in the proof of Lemma \ref{lem2}, we show that 
$$\frac{d}{dt}\Big|_{t=0}e^{-\frac{c}{2}\int_{\mathbb{R}^d}\ln(1-4t\bar{f_i}(s)f_j(s))ds}=c\langle f_i,f_j\rangle.$$
Finally, by using (\ref{deriv}), one gets
\begin{eqnarray*}
\frac{d}{dt}\Big|_{t=0}\|\alpha_1\Psi(\sqrt{t}f_1)+\dots+\alpha_n\Psi(\sqrt{t}f_n)\|^2&=&c\sum_{i,j=1}^n\bar{\alpha_i}\alpha_j\langle f_i,f_j\rangle\\
&=&c\|\alpha_1f_1+\dots+\alpha_nf_n\|^2.
\end{eqnarray*}
\end{proof}

As a consequence of the above lemma, we prove the following.
\begin{proposition}
Let $T$ be a contraction on $L^2(\Bbb R^d)\cap L^\infty(\Bbb R^d)$ with respect to the norm $\|.\|_\infty$. If $\Gamma_2(T)$ is a contraction on $\Gamma_2(L^2(\Bbb R^d)\cap L^\infty(\Bbb R^d))$, then $T$ is a contraction on $L^2(\Bbb R^d)\cap L^\infty(\Bbb R^d)$ with respect to the norm $\|.\|_2$.
\end{proposition}
\begin{proof}
We have $\Gamma_2(T)$ is a contraction on $\Gamma_2(L^2(\Bbb R^d)\cap L^\infty(\Bbb R^d))$. Then, for all $\alpha_1,\dots,\alpha_n\in\mathbb{C}$ and for all quadratic exponential vectors $\Psi(\sqrt{t}f_1),\dots,\Psi(\sqrt{t}f_n)$, such that $0\leq t\leq\delta$ and $\sqrt{\delta}\sup_{1\leq i\leq n}\|f_i\|_\infty<\frac{1}{2}$, one has 
$$\|\Gamma_2(T)(\alpha_1\Psi(\sqrt{t}f_1)+\dots+\alpha_n\Psi(\sqrt{t}f_n))\|^2\leq\|\alpha_1\Psi(\sqrt{t}f_1)+\dots+\alpha_n\Psi(\sqrt{t}f_n)\|^2$$
But, $\|\Gamma_2(T)(\alpha_1\Psi(\sqrt{t}f_1)+\dots+\alpha_n\Psi(\sqrt{t}f_n)\|^2$ is equal to
$$\|\alpha_1\Psi(\sqrt{t}T(f_1))+\dots+\alpha_n\Psi(\sqrt{t}T(f_n))\|^2$$
This yields
\begin{eqnarray*}\label{contrac}
\|\alpha_1\Psi(\sqrt{t}T(f_1))+\dots+\alpha_n\Psi(\sqrt{t}T(f_n))\|^2\leq\|\alpha_1\Psi(\sqrt{t}f_1)+\dots+\alpha_n\Psi(\sqrt{t}f_n\|^2
\end{eqnarray*}
Put
\begin{eqnarray*}
h_1(t)&=&\|\alpha_1\Psi(\sqrt{t}T(f_1))+\dots+\alpha_n\Psi(\sqrt{t}T(f_n))\|^2,\\
h_2(t)&=&\|\alpha_1\Psi(\sqrt{t}f_1)+\dots+\alpha_n\Psi(\sqrt{t}f_n)\|^2.
\end{eqnarray*}
It is clear that, for all $0\leq t\leq\delta$, $0\leq h_1(t)\leq h_2(t)$ and $h_1(0)=h_2(0)$. It follows that
$$\lim_{t\to0^+}\frac{h_1(t)-h_1(0)}{t}\leq \lim_{t\to0^+}\frac{h_2(t)-h_2(0)}{t}$$
Note that from Lemma \ref{lem4} one has
\begin{eqnarray*}
\lim_{t\to0^+}\frac{h_1(t)-h_1(0)}{t}&=&c\|T(\alpha_1f_1+\dots+\alpha_n f_n)\|^2  \\
\lim_{t\to0^+}\frac{h_2(t)-h_2(0)}{t}&=&c\|\alpha_1f_1+\dots+\alpha_n f_n\|^2
\end{eqnarray*}  
This ends the proof.
\end{proof}
\\ \\
\bigskip
{\bf Counter-example:} 

The following example proves that the condition 2) is not satisfied in the quadratic case. 
Define the operator $T$ on $L^2(\Bbb R)\cap L^\infty(\Bbb R)$ by
$$(Tf)(x)=f(2x),\;\;\forall x\in\mathbb{R}$$
It is clear that $T$ is a contraction on both $L^2(\Bbb R)$ and $L^\infty(\Bbb R)$ which is an homomorphism of $L^2(\Bbb R)\cap L^\infty(\Bbb R)$. Then, Proposition $4$ in \cite{AcDh1} shows that $\Gamma_2(T)$ is a contration on $\Gamma_2(L^2(\Bbb R)\cap L^\infty(\Bbb R))$. Moreover, a straightforward compuation shows that
$$(T^*f)(x)=\frac{1}{2}f(\frac{x}{2}),$$
for all $f\in L^2(\Bbb R)\cap L^\infty(\Bbb R)$. Note that $T^*$ takes the form $\psi T_1$ where $\psi=\frac{1}{2}$, with $\|\psi\|_\infty\leq1$, and $T_1$ is an homomorphism of $L^2(\Bbb R)\cap L^\infty(\Bbb R)$. Then, in the same way, Proposition $4$ in \cite{AcDh1} implites that $\Gamma_2(T^*)$ is a contraction on $\Gamma_2(L^2(\Bbb R)\cap L^\infty(\Bbb R))$. But, for all quadratic exponential vectors $\Psi(f)$ and $\Psi(g)$, one has
\begin{eqnarray*}
\langle \Gamma_2(T)\Psi(f),\Psi(g)\rangle&=&e^{-\frac{c}{2}\int_\mathbb{R}\ln(1-4\bar{f}(2s)g(s))ds}\\
&=&e^{-\frac{c}{4}\int_\mathbb{R}\ln(1-4\bar{f}(s)g(\frac{s}{2}))ds}\\
\langle \Psi(f),\Gamma_2(T^*)\Psi(g)\rangle&=&e^{-\frac{c}{2}\int_\mathbb{R}\ln(1-2\bar{f}(s)g(\frac{s}{2}))ds}
\end{eqnarray*}
In order to give more explaination for justifying why $\langle \Gamma_2(T)\Psi(f),\Psi(g)\rangle\neq \langle \Psi(f),\Gamma_2(T^*)\Psi(g)\rangle$ for all $f,g\in L^2(\Bbb R)\cap L^\infty(\Bbb R)$, remember that from the definition of the quadratic exponential vectors and proposition \ref{prop1} the scalar product $\langle \Gamma_2(T)\Psi(f),\Psi(g)\rangle$ is equal to 
\begin{eqnarray}\label{TT}
\sum_{n\geq0}\sum_{i_1+2i_2+\dots+ki_k=n}\frac{2^{2n-1}c^{i_1+\dots+i_k}}{i_1!\dots i_k!2^{i_2}\dots k^{i_k}}\langle T(f),g\rangle^{i_1}\langle (T(f))^2,g^2\rangle^{i_2}\dots\langle (T(f))^k,g^k\rangle^{i_k}
\end{eqnarray}
Because $(T(f))^m=T(f^m)$ for all $m\geq1$, the term in (\ref{TT}) is equal to
\begin{eqnarray*}
&&\sum_{n\geq0}\sum_{i_1+2i_2+\dots+ki_k=n}\frac{2^{2n-1}c^{i_1+\dots+i_k}}{i_1!\dots i_k!2^{i_2}\dots k^{i_k}}\langle T(f),g\rangle^{i_1}\langle T(f^2),g^2\rangle^{i_2}\dots\langle T(f^k),g^k\rangle^{i_k}\\
&&\!\!\!\!\!\!\!\!\!\!\!=
\sum_{n\geq0}\sum_{i_1+2i_2+\dots+ki_k=n}\frac{2^{2n-1}c^{i_1+\dots+i_k}}{i_1!\dots i_k!2^{i_2}\dots k^{i_k}}\langle f,T^*(g)\rangle^{i_1}\langle f^2,T^*(g^2)\rangle^{i_2}\dots\langle f^k,T^*(g^k)\rangle^{i_k}
\end{eqnarray*}
But for all $k\geq 2$ and for all function $g\in L^2(\Bbb R)\cap L^\infty(\Bbb R)$ such that $g\neq 0 \;a.e$, one has
 $$T^*(g^k)(x)=\frac{1}{2}g^k(\frac{x}{2}),\;\;(T^*(g))^k=(\frac{1}{2})^k g^k(\frac{x}{2})$$
This proves that $\Gamma_2(T^*)\neq(\Gamma_2(T))^*$.

\newpage
\bigskip
{\bf\large Acknowledgments}\bigskip

I gratefully acknowledge stimulating discussions with Eric Ricard and Uwe Franz. I would like also to thank Rolando Rebolledo for his hospitality during my visit to ``Laboratorio de An\'alisis y Estoc\'astico'', Chile.

\end{document}